\documentclass[10pt]{article}
\usepackage{amsmath,amssymb,amsbsy,amsfonts,amsthm,latexsym,
                        amsopn,amstext,amsxtra,euscript,amscd}
\newtheorem{theorem}{Theorem}
\newtheorem{lemma}{Lemma}

\newtheorem{cor}{Corollary}
\newtheorem{prop}[lemma]{Proposition}

\newtheorem{question}{Open Question}

\def\C{{\mathbb C}}

\def\M{{\mathbb M}}
\def\N{{\mathbb N}}

\def\R{{\mathbb R}}

\def\Ker{{\mathrm{Ker}}}
\def\Gal{{\mathrm{Gal}}}

\def\\{\cr}
\def\({\left(}
\def\){\right)}
\def\[{\left[}
\def\]{\right]}
\def\<{\langle}
\def\>{\rangle}
\def\fl#1{\left\lfloor#1\right\rfloor}

\newcommand{\ignor}[1]{}

\def\Mat{\mathop{Mat}}
\def\gcd{\mathop{gcd}}
\def\rank{\mathop{rank}}

\def\span{\mathop{span}}
\def\Re{\mathop{Re}}
\def\1{\mathbf{1}}

\begin{document}
\title{Balanced $0,1$-words and the Galois group of $(x+1)^n-\lambda x^p$}
\author{L.Glebsky}
\maketitle  
\begin{abstract}
We study the number of $0,1$-words where the fraction of $0$ is ``almost'' fixed for
any initial subword. It turns out that this study use and reveal the structure of the 
Galois group (the monodromy group) of the polynomials  $(x+1)^n-\lambda x^p$. ($p$ is not 
necessary a prime here.)  
\end{abstract}
\section{Introduction and formulation of the main results}\label{Sec.introduction}

We need some notations. Let $w\in \{0,1\}^*$. By $|w|$ we denote the length of 
$w$ ($|w_1w_2\dots w_n|=n$). By $|w|_0$ we denote the number of zeros in $w$
($|w_1w_2\dots w_n|_0=|\{i\;:\;w_i=0\}|$). Similarly, $|w|_1$ is the number of $1$
in $w$. For $k\in\N$, $k\leq |w|$ let $w[:k]=w_1w_2\dots w_k$ be the prefix of
length $k$. Generally, $w[l:k]=w_l,w_{l+1},\dots,w_k$.  
For $n,r\in\N$, $0\leq\alpha\leq 1$ let 
$$
B_{n,\alpha,r}=
\{w\in\{0,1\}^n\;\;:\;\;\forall k\leq n\;\;\alpha k-r<\left|\;w[:k]\;\right|_0
\leq\alpha k+r \}.
$$
The condition $\alpha k-r<|\,w[:k]\,|_0\leq \alpha k+r$ is equivalent to the condition 
$|\,w[:k]\,|_0-\fl{\alpha k}\in\{-r+1,-r+2,\dots,r\}$. 
We often use the last conditions as more manageable.
The elements of $B_{n,\alpha,r}$ is said to be $(\alpha,r)$-balanced words of length $n$.
We are interesting in $|B_{n,\alpha,r}|$, or, precisely, in growth exponent
\begin{equation}\label{Eq.limit}
e_{\alpha,r}=\lim_{n\to\infty} \sqrt[n]{|B_{n,\alpha,r}|}.
\end{equation}
In Section~\ref{Sec.rational} we calculate $B_{n,\alpha,r}$ for rational $\alpha$. It implies 
the existence of the limit (\ref{Eq.limit}) for rational $\alpha$. 
In Section~\ref{Sec.limit} we prove that the limit exists for all $\alpha\in (0,1)$ and 
that $e_{\alpha,r}$ 
is continuous in $\alpha$, uniformly with
respect to $r$.

Let us define 
$$
\tilde B_{n,\alpha,r}=
\{w\in\{0,1\}^n\;\;:\;\;\alpha n-r<|w|_0\leq\alpha n+r \}.
$$
It follows from study of generating function (as in 
\cite{P1,P2,Pemantle_preprint}) that
the growth exponent
$$
\tilde e_{\alpha}=\lim_{n\to\infty} \sqrt[n]{|\tilde B_{n,\alpha,r}|}.
$$
is independent of $r$ and 
$\tilde e_\alpha=\left(\frac{1}{\alpha}\right)^\alpha\left(\frac{1}{1-\alpha}\right)^{1-\alpha}$.
It is obvious that $B_{n,\alpha,r}\subset \tilde B_{n,\alpha,r}$ and 
\begin{equation}\label{ineq1}
e_{\alpha,r}\leq\tilde e_\alpha.
\end{equation}
In the paper we show by calculation that 
\begin{equation}\label{Eq.limit2}
\lim\limits_{r\to\infty}e_{\alpha,r}=\tilde e_\alpha 
=\left(\frac{1}{\alpha}\right)^\alpha\left(\frac{1}{1-\alpha}\right)^{1-\alpha}.
\end{equation}
 I was not able to find direct 
combinatorial argument for this limit.

It is interesting that there is a  relation of our results with Galois group 
$\Gal(P,\C(\lambda))$ of
polynomial $P=(x+1)^n-\lambda x^p$.  
In order to establish Limit~(\ref{Eq.limit2}) we use the fact that  
$\Gal(P,\C(\lambda))$ contains a cyclic permutation of length $n$.
Then, using the combinatorial inequality (\ref{ineq1}), we show that
$\Gal(P,\C(\lambda))=S_n$ for relatively prime $n$ and $p$. Let us finish the introduction
by repeating  the main results of the paper:
\begin{itemize}
\item The convergence of r.h.s of Eq.\ref{Eq.limit} is proved.
\item  The limit of Eq.\ref{Eq.limit2} is proved.
\item The equality $\Gal(P,\C(\lambda))=S_n$ is proved for relatively prime $n$ and $p$.
If $\gcd(n,p)=k$ then $(S_{n/t})^t$ is a normal subgroup of $\Gal(P,\C(\lambda))$ and
$\Gal(P,\C(\lambda))/(S_{n/t})^t$ is a cyclic group of order $t$.
\end{itemize}  

The questions discussed in the paper appear during the investigation of directional 
complexity and entropy for lift mappings initiated by V. Afraimovich and M. Courbage.
The author is thankful to them for the problem and useful discussions. As a reader
may see we don't touch the lift mappings in the present paper for it will be discussed
in \cite{Afr}. I believe also that the combinatorial problem is interesting for its own sake.

\section{Estimation of $|B_{n,\alpha,r}|$ for rational $\alpha$}\label{Sec.rational}
   
In order to calculate $|B_{n,\alpha,r}|$ we define a vector
$b(n)=(b_{1},b_{2},\dots,b_{2r})^t\in\R^{2r}$, 
$b_j(n)=|\{w\in B_{n,\alpha,r}\;\;:\;\; |w|_0-\fl{\alpha n}=j-r\}|$. Clearly, 
$|B_{n,\alpha,r}|=\sum\limits_j b_j(n)$ and $b_r(0)=1$, $b_j(0)=0$ for $j\neq r$.
Let
$$
N_+=\left( \begin{array}{ccccc}
          0 & 0 & \dots & \dots & 0\\
          1 & 0 & \dots & \dots & 0\\
          0 & 1 &  0    & \dots & 0\\
       \dots&\dots & \dots& \dots & \dots \\ 
       \dots&\dots & \dots& \dots & \dots \\ 
       \dots&\dots & \dots& 1 & 0\\
    \end{array}
   \right),\;\;\;N_-=N_+^t 
$$
be a low and upper $0$-Jordan cells and $E$ be the unit matrix. 
One may check that $b(n+1)=(E+N_+)b(n)$, if  $\fl{\alpha (n+1)}=\fl{\alpha n}$, and
$b(n+1)=(E+N_-)b(n)$, if  $\fl{\alpha (n+1)}=\fl{\alpha n}+1$.
So, at list for a rational $\alpha=p/q$ the problem of finding $e_{p/q,r}$ may be reduced 
to finding the maximal eigenvalue $e_{\max}$ of $M_{p+q}$, where 
$M_0=E$ and $M_{n+1}=(E+N_+)M_n$, if  $\fl{\alpha (n+1)}=\fl{\alpha n}$, and
$M_{n+1}=(E+N_-)M_n$, if  $\fl{\alpha (n+1)}=\fl{\alpha n}+1$. 
Notice, that the matrix $M_{p+q}$ is nonnegative irreducible primitive matrix\footnote{
It is an oscillation matrix \cite{Gant} and all its eigen-values are simple and positive}
(if $p,q\neq 0$). So, 
the Perron–-Frobenius theorem implies that 
\begin{equation}\label{Eq.value}
e_{\max}=e_{\alpha,r}^{p+q}.
\end{equation}

Let $e_1=(1,0,\dots,0)^t$, $e_2=(0,1,\dots,0)^t$,...,$e_j=(0,\dots,0,1,0,\dots)^t$ (the one on 
$j$-th place). One can check that $N_- e_j=e_{j-1}$, for $j>1$ and $N_+ e_j=e_{j+1}$, for
$j<2r$. Now notice, that for $p<j\leq 2r-q$ one has  
$N_-^{p_1}N_+^{q_1}\dots N_-^{p_k}N_+^{q_k}e_{j}=e_{j-p+q}$, where 
$p=p_1+p_2+\dots+p_k$ and $q=q_1+\dots+q_k$. So, the result is independent of the exact order
$N_-$ and $N_+$ (the matrices $N_-$ and $N_+$ almost commute in some sense.)    
We know that
\begin{equation}\label{Eq.M1}
M_{p+q}=
(E+N_-)^{p_1}(E+N_+)^{q_1}\dots (E+N_-)^{p_k}(E+N_+)^{q_k}.
\end{equation}
It follows that for 
$p<j\leq 2r-q$ the coefficients $M_{k,j}$  depend only on 
$p=p_1+\dots+p_k$ and $q=q_1+\dots+q_k$. 
Notice that $M_{k,j}=0$ if 
$j\not\in\{k-q,\dots,k+p\}$. Put $n=p+q$.   
So, there exist two submatrices $S_u\in \Mat_{q\times n}$ and $S_d\in\Mat_{p\times n}$ whose
coefficients does depend on the exact order of the multipliers in Eq.\ref{Eq.M1}.
A simple calculation shows $M_{k,j}=C^{p+q}_{k-j+p}$ if 
$(k,j)\not\in \{1,\dots,q\}\times\{1,\dots,n\}\cup \{2r-p,\dots,2r\}\times\{2r-n,\dots,2r\}$.
So the matrix $M$ looks like 
\def\M0{ \begin{array}{ccc} 0& \dots & 0\\
                                                 \dots   & \dots & \dots\\
                                                   0     &\dots  & 0
                                      \end{array} }
$$
\left( \begin{array}{ccccc} S_u & \dots & 0 &\dots & \M0 \\
                          \M0 &  \dots & C^{p+q}_{k-j+p} &\dots & \M0 \\
                          \M0 & \dots &  0 & \dots & S_d
\end{array}\right)
$$
Recall that the matrix $M$ depends on $p,q$ and $r$ and $M\in\Mat_{2r\times 2r}$. Some times we write this dependence explicitly 
(for example, $M(r)$ ). Recall that $n=p+q$.
\begin{theorem}\label{Th.spectrum}
Let $0\leq\lambda_1<\lambda_2\leq \frac{n^n}{p^pq^q}$. Then the number of spectral points
of $M(r)$ on $(\lambda_1,\lambda_2)$ goes to infinity as $r\to\infty$.
\end{theorem}
The theorem immediately imply 
\begin{cor}
Let $\alpha(r)$ be the maximal eigenvalue of $M(r)$. Then 
$\inf\lim \alpha(r)\geq \tilde e_{\alpha}^n$. With Inequality~(\ref{ineq1}) and Eq.\ref{Eq.value}
it proves
Limit~(\ref{Eq.limit2}) for rational $\alpha$. 
\end{cor}
  
In the next section we show how to manage irrational $\alpha$. We postpone the proof
of Theorem~\ref{Th.spectrum} up to Section~\ref{Sec.proof1} as we need some technical
results.

\section{Estimates for irrational $\alpha$}\label{Sec.limit}

Let us start listing some simple facts about $B_{n,\alpha,r}$.
\begin{itemize}
\item If $w\in B_{n,\alpha,r}$ then $\{w0,w1\}\cap B_{n+1,\alpha,r}\neq\emptyset$. 
(The prolongation property)
\item If $n_1\leq n_2$ and $r_1\leq r_2$ then $|B_{n_1,\alpha,r_1}|\leq |B_{n_2,\alpha,r_2}|$.
\item $|B_{n,\alpha,r}|=|B_{n,1-\alpha,r}|$. Since there exists bijection between $|B_{n,\alpha,r}|$ and $|B_{n,1-\alpha,r}|$
induced by $0\leftrightarrow 1$. 
\end{itemize}
\begin{prop}\label{Prop.estimates} 
There exist $K_n(\alpha,\alpha'):\N\times (0,1)\times (0,1)\to (1,\infty)$  
$$
\lim\limits_{\alpha'\to\alpha}(\lim_{n\to\infty}\sqrt[n]{K_n(\alpha,\alpha')})= 1
$$ 
and $K_n(\alpha,\alpha')^{-1}|B_{n,\alpha',r}|\leq |B_{n,\alpha,r}|\leq 
K_n(\alpha,\alpha') |B_{n,\alpha',r}|$.
\end{prop}
From the proposition follows that if limit~(\ref{Eq.limit}) exists for all rational 
$\alpha\in (0,1)$ then it exists for all
$\alpha\in (0,1)$. Moreover, the resulting family of functions $e_{\alpha,r}$ is uniformly  
continuous. So, 
$e_\alpha=\lim\limits_{r\to\infty}e_{\alpha,r}$ is continuous function. 
(Also, it is enough to check the convergence only for rational $\alpha$.)
\begin{proof}
Let $\alpha<\alpha'$. We start by constructing a map $\psi:B_{n,\alpha,r}\to B_{n',\alpha',r}$ 
where $n'$ will be specified later. 
To define $\psi(w)$ we add some $0$ between letters of $w$ by the following procedure.
\bigskip

{\it start $w$, $w'=\emptyset$ ($w'=$empty word)

step=0

while($w\neq\emptyset$):

\hspace{10.pt} step=step+1

\hspace{10.pt} $w_1$ is the first letter of w

\hspace{10.pt} $w=w[2:]$ \# remove the first letter from $w$

\hspace{10.pt} $w'=w'0^kw_1$, where $k$ is minimum $k$, such that 
$w'0^kw_1\in B_{|w'|+k+1,\alpha',r}$

end while
}
\bigskip

\noindent We have to check that the algorithm works. The only possible problem is in the last 
operation inside {\it while}. 
We prove that such $k$ exists by induction on {\it step}. 
Let $w$ be an initial word and $w'(s)$ be the word 
$w'$ after $s$ steps of algorithm. 

\bigskip
\noindent{\bf Statement} Let $w'(s)$ exist $s<|w|$ and 
$|\,w[:s]\,|_0-\fl{\alpha s}\geq |\,w'(s)\,|_0-\fl{\alpha' |w'(s)|}$. 
Then $w'(s+1)$ exists and
 $|\,w[:s+1]\,|_0-\fl{\alpha (s+1)}\geq |\,w'(s+1)\,|_0-\fl{\alpha' |w'(s+1)|}$.
\bigskip

\noindent Indeed, $|\,w[:s+1]\,|_0-\fl{\alpha (s+1)}\geq |\,w'(s)w_{s+1}\,|_0-
\fl{\alpha' (|w'(s)|+1)}$.
(Recall that $\alpha<\alpha'$.) 
So, if $w'(s)w_{s+1}\in B_{|w'(s)|+1,\alpha',r}$ we are done. 
If $w'(s)w_{s+1}\not\in B_{|w'(s)|+1,\alpha',r}$ then
 $-r+1 >|\,w'(s)w_{s+1}\,|_0-\fl{\alpha' (|w'(s)|+1)}$. Particularly, it implies that 
$w_{s+1}=1$. By the prolongation property $w'(s)0\in B_{|w'(s)|+1,\alpha',r}$  and there exists 
minimal $k$ such that 
$-r+1= |\,w'(s)w_{s+1}\,|_0+k-\fl{\alpha' (|w'(s)|+k+1)}$. It is easy to check that this $k$ 
is the $k$ from the last command of {\it wile} for the
{\it step}$=s+1$. The inequality of the Statement holds because $-r+1$ is the minimum possible 
value for $|\,w[:s]\,|_0-\fl{\alpha s}$.

Now, after $n=|w|$ steps we have $w'\in B_{|w'|,\alpha',r}$. This $w'$ is made from $w$ by adding 
somewhere $j$ zeroes. The statement allows us to estimate $j$:
$|\,w\,|_0-\fl{\alpha n}\geq |\,w\,|_0+j-\fl{\alpha'(n+j)}$, and 
$j\leq (1-\alpha')^{-1}(\alpha' n-\fl{\alpha n})$. Let
$j_{\max}= \fl{(1-\alpha')^{-1}(\alpha' n-\fl{\alpha n})}$ and
$n'=n+j_{\max}$. Now we define 
$\psi:B_{n,\alpha,r}\to B_{n',\alpha',r}$ as follows
$\psi(w)=w'u$ where $w'u \in B_{n',\alpha',r} $,  and $u$, say, lexicographically minimal 
satisfying this condition.

Let $w \in B_{n',\alpha',r}$. One has  
$$
|\psi^{-1}(w)|\leq \sum_{j=0}^{j_{\max}}C_j^{\alpha'n'+r}.
$$  
Indeed, any word from $\psi^{-1}(w)$ is made by removing a suffix of length $\leq j_{\max}-j$ 
and then by removing $j$ zeros from the rest. Now, we estimate 
$$
|B_{n,\alpha,r}|\leq \left(\sum_{j=0}^{j_{\max}}C_j^{\alpha'n'+r}\right)|B_{n',\alpha',r}|
\leq K_n\cdot |B_{n,\alpha',r}|,
$$
where $K_n=2^{j_{\max}}\sum_{j=0}^{j_{\max}}C_j^{\alpha'n'+r}$. Now,
$$
\sqrt[n]{K_n}\leq 2^{\frac{j_{\max}}{n}}\sqrt[n]{n}\sqrt[n]{C_{j_{\max}}^{\alpha'n'+r}},
$$ 
but $\frac{j_{\max}}{n}\to \frac{\alpha'-\alpha}{1-\alpha'}$, $\sqrt[n]{n}\to 1 $  and 
$\sqrt[n]{C_{j_{\max}}^{\alpha'n'+r}}\to \left(\frac{1}{\beta}\right)^\beta\left(\frac{1}{1-\beta}\right)^{1-\beta}$
when $n\to\infty$. Where $\beta=\lim\limits_{n\to\infty} \frac{j_{\max}}{\alpha'n'+r}=\frac{\alpha'-\alpha}{1-\alpha}$. 
In order to find $\lim\limits_{n\to\infty} \sqrt[n]{C_{j_{\max}}^{\alpha'n'+r}}$ one can use the methods of \cite{P1} as it is explained 
in Section~\ref{Sec.asymptotic}. 
For inequality from over direction we change $\alpha\to 1-\alpha$ and $\alpha'\to 1-\alpha'$.
Generally, we get another $K_n$, but then we take maximum of them. 
\end{proof} 

\section{About intersections of  linear subspaces.}\label{Sec.subspace}

For the proof of Theorem~\ref{Th.spectrum} I need several lemmas about linear spaces. 
In all these lemmas we use an extension of linear space to a larger field. Let $V$ be
an $n$-dimension vector space over a field $K$. Let $F$ be a finite extension of $K$.
Then $F\otimes_K V$ is an $n$-dimension vector space over $F$. Clearly, if $\tilde S$ is a
$p$-dimension subspace of $V$ then $F\otimes_K \tilde S$ is a $p$-dimension subspace 
(as a vector space over $F$) of $F\otimes_K V$. Let $\alpha\in\Gal(F:K)$. $\alpha$ may be 
extended to $\alpha\otimes id: F\otimes_KV\to F\otimes_KV$. One can see that $\alpha\otimes id$ is not
linear over $F$, but as we will see it conserve linear dependence. 
In what follows I write $\alpha$  instead of  $\alpha\otimes id$ 
in order to simplify notations. So, for example, $\alpha(f\otimes v)=\alpha(f)\otimes v$. 
Let $\bar x_1\in F\otimes_K V$, 
$\bar x_2=\alpha(\bar x_1),\;\bar x_3=\alpha(\bar x_2),\dots,\bar x_{j+1}=\alpha(\bar x_j),\dots$.
\begin{lemma}\label{lemma.0}
Suppose that $\bar x_1, \bar x_2,\dots,\bar x_n$ form a basis of $F\otimes_K V$ over $F$. Then
$$
\span(\bar x_1,\bar x_2,\dots,\bar x_{n-p})\cap F\otimes_K {\tilde S}=\{0\} 
$$ 
\end{lemma}       
\begin{proof}
Notice the following simple fact:
\begin{itemize}
\item $\alpha(F\otimes_K {\tilde S})=F\otimes_K{\tilde S}$.
\item Let $y\in F\otimes_K V$ and 
$\alpha^k(y)\in\span\left(y,\alpha(y),\alpha^2(y),\dots,\alpha^{k-1}(y)\right)$.
\newline
Then
$\alpha^m(y)\in\span\left(y,\alpha(y),\alpha^2(y),\dots,\alpha^{k-1}(y)\right)$ for 
any $m\in\N$.
\end{itemize}
We prove the following inductive

\bigskip
\noindent{\bf Statement} Let $k\leq n-p$ and $\span(\bar x_1,\bar x_2,\dots,\bar x_k)\cap F\otimes_K{\tilde S}\neq\{0\}$.
Then $\span(\bar x_1,\dots,\bar x_{2k+p-n-1})\cap F\otimes_K{\tilde S}\neq\{0\}$.
(We suppose here that $\span(\bar x_1,\dots,\bar x_m)=\span(\emptyset)=\{0\}$ for $m<1$. In this case the statement leads
to a contradiction that proves our lemma.) 
\bigskip

\noindent Indeed, let $0\neq z\in\span(\bar x_1,\dots,\bar x_k)\cap F\otimes_K{\tilde S}$. Then
$0\neq\alpha(z)\in\span(\bar x_2,\dots,\bar x_{k+1})\cap F\otimes_K{\tilde S}$,...,
$0\neq\alpha^{p-1}(z)\in\span(\bar x_p,\dots,\bar x_{k+p-1})\cap F\otimes_K{\tilde S}$,....,
$0\neq\alpha^{n-k}(z)\in\span(\bar x_{n-k+1},\dots,\bar x_{n})\cap F\otimes_K{\tilde S}$. As the dimension of 
$F\otimes_K{\tilde S}$ is $p$ one has that 
$\alpha^p(z)\in\span\left(z,\alpha(z),\dots,\alpha^{p-1}(z)\right)$
and, consequently, $\alpha^{n-k}(z)\in\span\left(z,\alpha(z),\dots,\alpha^{p-1}(z)\right)$. 
It follows that $\alpha^{n-k}(z)\in\span(\bar x_1,\dots, \bar x_{k+p-1})\cap 
\span(\bar x_{n-k+1},\dots,\bar x_n)$ and $\alpha^{n-k}(z)\in\span(\bar x_{n-k+1},\dots,\bar x_{k+p-1})$, as 
$\bar x_i$ form a basis. Applying $\alpha^{-n+k}$ to the last inclusion one gets
$z\in\span(\bar x_1,\dots,\bar x_{2k+p-n-1})$.
Obviously, $2k+p-n-1<k$ for $k<n-p+1$ and the lemma is proved.      
\end{proof}

The next lemma is an easy corollary of Lemma~\ref{lemma.0}.
\begin{lemma}\label{lemma.1}
Let $P\in\C(\lambda)[x]$ of order $n$, $x_i$ be the roots of $P$ and 
the cyclic permutation $(x_1,x_2,\dots,x_n)\in\Gal(P,\C(\lambda))$.
Let $S\in \Mat_{p,n}(\C(\lambda))$ with $\rank(S)=p$. Let
$$
X=\left(\begin{array}{cccc} 1    &    1    & \dots    & 1\\
                           x_1   &   x_2    & \dots   & x_p\\
                           x_1^2 & x_2^2    & \dots   & x_p^2 \\
                           .     &    .    & \dots    & . \\
                           .     &    .    & \dots    & . \\
                           .     &    .    & \dots    & . \\
                        x_1^{n-1}& x_2^{n-1} &\dots & x_p^{n-1}
              \end{array}\right)
$$     
Then equation $\det(SX)=0$ (as the function of $\lambda$)  has finite solutions in each 
compact subset of the corresponding Riemann surface.  
\end{lemma}
\begin{proof}
Suppose the contrary, that $\det(SX)=0$ has a countable set of solutions  in
a compact subset of $\C$.  It follows that $\det(SX)\equiv 0$ 
for all $\lambda$. Now, the lemma follows by application of Lemma~\ref{lemma.0}
with $K=C(\lambda)$, $F$ being the splitting field of $P$, ${\tilde S}=\Ker(S)$ and
$\bar x_1=x_1^{[0:n-1]}=(1,x_1,\dots,x_1^{n-1})$. 
\end{proof}

Let $V=\R^n$ be a real linear space. We need the complexification $\C\otimes_\R V$ of $V$. 
On $\C\otimes_\R V$ the following antilinear involution is defined
$\cdot^*:\C\otimes_\R V\to \C\otimes_\R V$ as $(c\otimes v)^*=c^*\otimes v$. 
A subspace $Y$ of $\C\otimes_\R V$ is of the form $Y=\C\otimes_\R X$ if and only if 
$Y$ is closed with respect to involution $(\cdot)^*$. In this case we denote $X=\Re(Y)$.
 
Let $V_u$ be a subspace of $V$, $\dim(V_u)=q$. Let $L:V\to V$ be a real linear operator, 
diagonalizable in $\C\otimes_\R V$. 
Let $e_1,\dots e_n$ be eigen-vectors of $L$ with corresponding 
eigen-values $\alpha_1,\dots, \alpha_n$, ordering such that 
$|\alpha_1|\geq |\alpha_2|\geq\dots\geq |\alpha_n|$.

\begin{lemma}\label{lemma.2}
Suppose that
\begin{itemize}
\item $\alpha_q=\rho e^{i\phi}\not\in\R$, $\alpha_{q+1}=\alpha_q^*$ and 
$|\alpha_{q+1}|>|\alpha_{q+2}|$.
\item $\span(e_{q+1},\dots e_n)$ and $\C\otimes_\R V_u$ are in general position.
\end{itemize}
Then   $L^d(V_u)\to \span(\Re\left(\span(e_1,e_2,...,e_{q-1})\right),ae^{id\phi}e_q+a^*e^{-id\phi}e_{q+1})$ as 
$d\to\infty$.
Precisely, $L^d(V_u)$ has a basis 
$b_1+v_1,\dots, b_{q-1}+v_{q-1},ae^{in\phi}e_q+a^*e^{-in\phi}e_{q+1}+v_q$ with
$\|v_i\|\leq c\left(\frac{|\alpha_{q+2}|}{|\alpha_i|}\right)^d$. 
Where $c$ is independent of $d$ and $b_1,\dots b_{q-1}$ is a (real) basis of $\Re\left(\span(e_1,e_2,...,e_{q-1})\right)$. 
\end{lemma}

\begin{proof}

One can choose $e_j$ such that $e_j^*=e_{\tilde j}$, where $\tilde\cdot:\{1,2,\dots,n\}\to \{1,2,\dots,n\}$. 
Moreover, $\tilde j\in \{1,\dots,q\}$ for 
$j\in \{1,\dots,q\}$ and  $e_q^*=e_{q+1}$.
    
Let $V^\perp$ ($(\C\otimes_\R V)^\perp$) denote the dual space of $V$ over $\R$ ($\C\otimes_\R V$ over $\C$). 
Clearly, $V^\perp\subset (\C\otimes_\R V)^\perp$ and $f\in V^\perp$ if and only if 
$f(v^*)=(f(v))^*$ for any $v\in \C\otimes_\R V$.
Let $f_1,f_2,\dots,f_{n-q}\in V^\perp$ be such that 
$v\in V_u\;\; \Longleftrightarrow\;\; \forall j\in\{1,2,\dots,n-q\}\;f_j(v)=0$.
Let $(j,k)\in \{1,2,\dots,n-q\}^2$. One has that $\det(f_j(e_{q+k}))\neq 0$ 
(the condition of the general position). It follows that $\C\otimes_\R V_u$ has a
basis $e_j+w_j$ where $j\in \{1,\dots,q\}$ and 
$w_j\in\span(e_{q+1},\dots,e_n)$, $w_j^*=w_{\tilde j}$. 
Now, for $j=1,2,\dots, q-1$ one has that $w_j\in\span(e_{q+2},\dots,e_n)$ and $w_q=e_{q+1}+\dots$. 
In order to finish the
proof one should apply $L^d$ to the basis and change a pair $e_j$, $e_{\tilde j}$ by 
$b_j=(e_j+e_{\tilde j})/2$ and
$b_{\tilde j}=(e_j-e_{\tilde j})/2i$.
\end{proof}

I need a parametric variant of the presiding lemma. Now our space $V_u$ and 
linear operator 
$L$ are $C_1$-smoothly depend on
a real parameter $\lambda$, $\dim(V_u)=q$. All eigen-values 
$\alpha_i$ of $L$ are different for 
$\lambda=\lambda_0$ and ordered as before. 
Also we have another space $V_d$, $\dim(V_d)=n-q$. $V_d$ as well $C_1$-smoothly depends on
$\lambda$.

\begin{lemma}\label{lemma.3} 
Let for $\lambda_0$ one have
\begin{itemize}
\item  $\alpha_q=\rho e^{i\phi}\not\in\R$, $\alpha_{q+1}=\alpha_q^*$ and 
$|\alpha_{q+1}|>|\alpha_{q+2}|$; 
\item $\frac{d\phi}{d\lambda}\neq 0$
\item  $\span(e_{q+1},\dots e_n)$ and $\C\otimes_\R V_u$ are in general position
\item $\span(e_1,\dots e_q)$ and $\C\otimes_\R V_d$ are in general position. 
\end{itemize}  
Then for any interval $(\lambda_1,\lambda_2)\ni\lambda_0$ there exists $d_0$ such that for any 
$d_0<d\in\N$ one has
$L^d(V_u)\cap V_d\neq\{0\}$ for some $\lambda\in (\lambda_1,\lambda_2)$.
\end{lemma}

\begin{proof}
Fix a basis $ w_1,w_2,\dots,w_{n-q}$ of $V_d$. Let $\det(w_1\dots,w_{n-q},b_1,\dots,b_{q-1},e_q)=A(\lambda)$ 
(I suppose that $V=R^n\subset\C^n$. So, each vector is a vector column. $b_i$ is from Lemma~\ref{lemma.2}.) 
One has that $A(\lambda_0)\neq 0$ (general position).
Let $L^d(V_u)=\span(B_d)$, where $B_d$ is a basis of Lemma~\ref{lemma.2}. It follows that
\begin{equation}\label{Eq.*}
\det(w_1,\dots,w_{n-q},B_d)=A(\lambda)e^{id\phi(\lambda)}+A^*(\lambda)e^{-id\phi(\lambda)}+C(\lambda,d)=
2|A(\lambda)|\cos(d\phi(\lambda)+\phi_0)+C(\lambda,d) 
\end{equation}
where $C(\lambda,d)\to 0$ when $d\to \infty$. The r.h.s. of Eq.\ref{Eq.*} is a fast oscillating function. So,
for any $\epsilon>0$, one can find $d_0\in\N$ such that for any $d>d_0$ r.h.s of Eq.(*) change the sign on
$(\lambda_0-\epsilon,\lambda_0+\epsilon)$. It follows that the l.h.s. determinant is $0$ for some  
$\lambda\in (\lambda_0-\epsilon,\lambda_0+\epsilon)$. The zeros of the determinant correspond to nontrivial intersections
of $L^d(V_u)$ and $V_d$.
\end{proof}
  
\section{Some property of $P=(x+1)^n-\lambda x^{p}$}\label{Sec.P}
Critical values of $P$.

\begin{prop}\label{Prop.critical}
The only critical values are $\lambda=0$ (with all roots $x=-1$) and
$\lambda=\frac{n^n}{p^p(n-p)^{(n-p)}}$ with the only double root
$x=\frac{p}{n-p}$ 
\end{prop}
\begin{proof}
The proof is a direct computation.
\end{proof}

\begin{prop} \label{Prop.small.roots}
There exists a converging around $0$ series $f$, such that the roots of P are
$x_i=\gamma_i+\gamma_i^2f(\gamma)-1$ where $\gamma_i=\sqrt[n]{(-1)^p\lambda}$ (the $i$-th root)
and $|\lambda|$ small enough ($\gamma$ are inside the convergence ball of $f$).  
\end{prop}
\begin{proof}
We apply the Newton method to calculate roots of $P$ for small  $|\lambda|$. 
\end{proof}

We are interesting in small positive $\lambda$. Let $\lambda=\epsilon^n$, 
$\epsilon>0$. Then $x_j=\epsilon\exp(2i\pi j/n)-1+O(\epsilon^2)$ for even $p$ and
$x_j=\epsilon\exp(2i\pi (j+1/2)/n)-1+O(\epsilon^2)$ for odd $p$. It follows from
Proposition~\ref{Prop.small.roots} that $x_j\to x_{j+1}$ if $\lambda$ rotates around $0$.
It means that Galois group $\Gal(P,\C(\lambda))$ (that is the monodromy group in this case) 
contains cycle $(x_1,x_2,\dots,x_n)$.

\begin{prop}\label{Prop.unique.mod}
Let $\lambda\in \R$, $P(x)=0$, $P(y)=0$ with $|x|=|y|$. Then $x=y$ or $x=y^*$. 
\end{prop} 
\begin{proof}
Indeed, $|x+1|^n=|\lambda|\cdot |x|^p=|y+1|^n$. It follows that $|x+1|^2=|y+1|^2$. Denoting
$x=\rho e^{i\phi}$ and $y=\rho e^{i\psi}$ one gets $\cos(\phi)=\cos(\psi)$. 
\end{proof}

Notice that $|x_j|^2=1-2\epsilon\cos(2\pi (j+(1-(-1)^k)/4)/n)+O(\epsilon^2)$. 
So, for small enough $\lambda=\epsilon^n$, one has the following ordering of $|x_j|$:
\begin{prop}\label{Prop.order}
\begin{itemize}
\item $|x_{\frac{n-1}{2}}|=|x_{-\frac{n-1}{2}}|>|x_{\frac{n-1}{2}-1}|=|x_{-\frac{n-1}{2}+1}|>\dots >|x_j|=|x_{-j}|>\dots >|x_0|$ if $p$ is even and $n$ is odd;
\item $|x_{\frac{n-1}{2}}|>|x_{\frac{n-1}{2}-1}|=|x_{-\frac{n-1}{2}}|>\dots >|x_j|=|x_{-j-1}|>\dots >|x_0|=|x_{-1}|$ if $p$ is odd and $n$ is odd; 
\item $|x_{\frac{n}{2}-1}=|x_{\frac{n}{2}}|>|x_{\frac{n}{2}-2}=|x_{-\frac{n}{2}+1}|>\dots >|x_j|=|x_{-j-1}|>\dots >|x_0|=|x_{-1}|$ if $p$ is odd and $n$ is even.
\end{itemize}
From Proposition~\ref{Prop.critical} and Proposition~\ref{Prop.unique.mod} it follows that this order is conserved for 
$\lambda\in (0,\frac{n^n}{p^p(n-p)^{(n-p)}})$.
\end{prop}
\begin{prop}\label{Prop.arg}
Let $x=\rho e^{i\phi}\not\in\R$ be a root of $P$, and $\lambda\in\R$. Then $\frac{d\phi}{d\lambda}\neq 0$. 
\end{prop}

\begin{proof}
$\frac{d\phi}{d\lambda}=0$ if and only if $\frac{1}{x}\frac{dx}{d\lambda}\in\R$. Direct calculations show
$$
\frac{dx}{d\lambda}=\frac{x(x+1)}{\lambda((n-p)x-p)},
$$
but $\frac{x+1}{\lambda((n-p)x-p)}\in\R$  if and only if $x\in\R$ (at least for $p\leq n$).
\end{proof}

Now we formulate our result on the Galois group of $P$. It will be proved in the next section.
\begin{theorem}\label{Th.galois}
Let $P=(x+1)^n-\lambda x^p$, $n>p$, $\gcd(n,p)=1$. Then $\Gal(P,\C(\lambda))=S_n$. Where $S_n$ is a group of all
permutations of $n$ elements.
\end{theorem}

\section{Proofs of Theorem~\ref{Th.spectrum} and Theorem~\ref{Th.galois}}\label{Sec.proof1}

We start with a proof of Theorem~\ref{Th.spectrum}.  We show that for any $0<\lambda_0<\frac{n^n}{p^pq^q}$
and any $\epsilon>0$ and large enough $r$ there exists a nontrivial solution to  
\begin{equation}\label{Eq.spectral}
Mv=\lambda v
\end{equation}
for 
$\lambda \in (\lambda_0-\epsilon,\lambda_0+\epsilon)$. We reduce this problem to the application of Lemma~\ref{lemma.3}.  
Define $x^{[0:2r-1]}=(1,x,...,x^{2r-1})^t$. Let $x_i$ be the roots of $(x+1)^n-\lambda x^p$. All of them are different if 
$0<\lambda<\frac{n^n}{p^pq^q}$. Notice that $(Mv-\lambda v)[q+1,2r-p]=0$ if and only if $v$ is a linear combinations of 
$x_i^{[0:2r-1]}$. 
So, we only need to
find $\lambda$ and linear combinations of $x_i^{[0:2r-1]}$ to satisfy the boundary conditions,
that is, to find $\alpha_1,\alpha_2,\dots,\alpha_n$ and $\lambda$ such that
$(S_u-\lambda)(\alpha_1x_1^{[0:n-1]}+\dots+\alpha_nx_n^{[0:n-1]})=0$ and
$(S_d-\lambda)(\sum x_i^{2r-n}\alpha_ix_i^{[0:n-1]})=0$. Where I use the inclusions $\R\to\Mat_{p,n}$  ( $\R\to\Mat_{q,n}$) 
defined as
$$
\lambda\to\left(\begin{array}{cccc} \lambda & 0 & \dots & 0\\
                                      0 &\lambda & \dots & 0\\
                                    \vdots & \vdots &\vdots &0\\
                \end{array}\right) 
$$  
One can consider $S_u-\lambda$ as a linear operator $\R^n\to\R^p$ and $S_d-\lambda$ as a linear operator
$\R^n\to\R^q$. Let the linear operator $L:\R^n\to\R^n$ be defined as $x_i^{[0:n-1]} \to x_i\cdot x_i^{[0,n-1]}$.
Let $V_u=\Ker(S_u-\lambda)$ and $V_d=\Ker(S_d-\lambda)$. 
Notice that $\rank(S_u-\lambda)=p$ ($\rank(S_d-\lambda)=q$) over $\C(\lambda)$. So, $\dim(V_u)=q$ ($\dim(V_d)=p$) for all
but finite exceptions $\lambda$.  
The following proposition is standard:
\begin{prop}\label{Prop.reduce}
Eq.\ref{Eq.spectral} has a nontrivial solution if and only if $L^d(V_u)\cap V_d\neq\{0\}$ with $d=2r-n$.
\end{prop}
So, in order to finish the proof we show that $L$, $V_u$, $V_d$ satisfy Lemma~\ref{lemma.3} for all, but some
finite exceptions, $0<\lambda_0<\frac{n^n}{p^pq^q}$.
\begin{itemize}
\item The first item is fulfilled by Proposition~\ref{Prop.order}.
\item The second item is fulfilled by Proposition~\ref{Prop.arg}.
\item The third and fourth items are fulfilled thanks to Lemma~\ref{lemma.1}. 
\end{itemize}

Now, let us consider Theorem~\ref{Th.galois}. By Inequality~(\ref{ineq1}), $M$ has no spectral values $>\frac{n^n}{p^pq^q}$.
It means, that $\lambda\geq \frac{n^n}{p^pq^q}$ the conditions of Lemma~\ref{lemma.3} are no more fulfilled. Using this one can easily
figure out that pair of roots collides for $\lambda=\frac{n^n}{p^pq^q}$. These are $(x_{\frac{p}{2}},x_{-\frac{p}{2}})$ for even $p$ and
$(x_{\frac{p-1}{2}},x_{-\frac{p-1}{2}-1})$ for odd $p$. It means that the  transposition  $(x_{\frac{p}{2}},x_{-\frac{p}{2}})\in\Gal(P,C(\lambda))$
($(x_{\frac{p-1}{2}},x_{-\frac{p-1}{2}-1})\in\Gal(P,C(\lambda))$). So, $\Gal(P,C(\lambda))$ (which is the same as the monodromy group) is generated by
a pair of permutations $(0,1,2,\dots,n-1)$ and $(a,a+p)$. One can check that this is $S_n$ if $\gcd(p,n)=1$. If $\gcd(p,n)=t$ then
$\Gal(P,C(\lambda))=\langle (1,2,\dots,n) (1,1+t)  \rangle$ is an extension of $C_t$ (the cyclic group of order $t$) by 
$S_{\frac{n}{t}}^t$.

\section{Asymptotic for 2-variable generating functions.} \label{Sec.asymptotic}
I rewrite Theorem~1.3 of \cite{Pemantle_preprint}, see also
\cite{P1,P2}

Let
$$
G(x,y)=\frac{P(x,y)}{D(x,y)}=\sum_{r=0,s=0}^\infty a_{rs}x^ry^s,\;\;
a_{rs}\geq 0
$$

\begin{theorem}\label{th_comb1}
\begin{enumerate}
\item For each positive $(r,s)$ (in the positive octant), there is a unique
positive solution $(x,y)$ of the system
\begin{eqnarray}
D &= & 0\\
s x\frac{\partial D}{\partial x} &=& r y\frac{\partial
D}{\partial y}
\end{eqnarray}
(clearly, the solution depends on the direction $s/r$, but not the
absolute value of $(s,r)$.)
\item With $(x,y)$ being the solution defined above,
if $P(x,y)\not=0$,
$$
a_{rs}\sim f_{rs}= \frac{P(x,y)}{\sqrt{2\pi}}x^{-r}y^{-s}
\sqrt{\frac{-yD_y}{sQ(x,y)}}
$$
uniformly over compact cones of $(r,s)$ for which $(x,y)$ is a
smooth point of the manifold $D=0$. (This means that if $(r,s)$ is
in the cone then $|1-a_{rs}/f_{rs}|\leq \epsilon(\sqrt{s^2+r^2})$
for some $\epsilon(n)$, $\epsilon(n)\to 0$ when $n\to\infty$.)
Where

\noindent 
$Q(x,y)=-xD_x(yD_y)^2-yD_y(xD_x)^2-[(yD_y)^2x^2D_xx+(xD_x)^2y^2D_yy-2xD_xyD_yxyD_{xy}]$

\end{enumerate}
\end{theorem}
Particularly, it implies that
$\lim\limits_{n\to\infty}\sqrt[n]{a_{\fl{\alpha n},\fl{(1-\alpha)n}}}=x^{-\alpha}y^{-(1-\alpha)}$. In our case
$$
G(x,y)=\sum_{r=0,s=0}^\infty C_{r}^{r+s}x^ry^s=\frac{1}{1-x-y},
$$
and $x(\alpha)=\alpha$ and $y(\alpha)=1-\alpha$.
\section{Concluding remarks}\label{Sec.remarks}

We may consider the similar problem in more general situation. Let $\Gamma$ be a graph.
Let $\Gamma=\Gamma_0\cup\Gamma_1$. We are interesting in paths $p$ of the graph. 
Let $|p|$ be the length of $p$ and $|p|_i$ be the number of transitions (edges) of $p$
from the graph $\Gamma_i$. Let $P$ be the set of paths of $\Gamma$ and
$$
P_{n,\alpha,r}=\{p\in P\;\;|\;\;|p|=n,\;\forall k \;\alpha k-r<|\,p[:k]\,|\leq\alpha k+r\}.
$$  

One may ask the same questions about $|P_{n,\alpha,r}|$. Some part of our consideration
may be easily generalized for this new situation. For example, the matrix $M$ become 
a product of $N_+\otimes A_1+E\otimes A_2$ and $E\otimes A_1+N_-\otimes A_2$, where $A_i$ is
the incidence matrix of $\Gamma_i$. 

In general, one can redefine $e_{\alpha,r}$ and $\tilde e_r$ for this situation.

\begin{question}
Is it true that for any $\Gamma=\Gamma_1\cup\Gamma_2$ one has that $\lim\limits_{r\to\infty} e_{\alpha,r}=\tilde e_r$.
\end{question}
If the answer is ``yes'' it may happens that it is more easy to find a combinatorial proof.
But I was not able to find a combinatorial proof even for the case considered in the present article.

Notice, that establishing our limits we use a fast oscillating function. It is interesting, that in the proof of Theorem~\ref{th_comb1}
the authors use a fast oscillating integrals. So, it may happens that there exists more deep relation between our calculations and
results of \cite{P1, P2, Pemantle_preprint}

\end{document}